\documentclass[11pt]{amsart}
\usepackage[margin=1in]{geometry}
\usepackage{amsmath}
\usepackage{amsfonts,amsmath}
\usepackage{paralist}
\usepackage[colorlinks=true]{hyperref}
\hypersetup{urlcolor=blue, citecolor=red}
 \numberwithin{equation}{section}

\newtheorem{theorem}{Theorem}[section]
\newtheorem{corollary}[theorem]{Corollary}

\newtheorem{lemma}[theorem]{Lemma}
\newtheorem{proposition}[theorem]{Proposition}

\theoremstyle{definition}
\newtheorem{definition}[theorem]{Definition}

\newcommand{\RN}{\mathbb{R}^d}

\newcommand{\ot}{\Omega_T }

\newcommand{\po}{\partial\Omega}
\newcommand{\mdiv}{\textup{div}}
\newcommand{\mdet}{\textup{det}}

\newcommand{\io}{\int_{\Omega}}
\newcommand{\ioT}{\int_{\Omega_{T}}}

\newcommand{\wk}{w^{(k)}}
\newcommand{\wkg}{\left(w^{(k)}\right)}

\newcommand{\vok}{v_1^{(k)}}
\newcommand{\vtk}{v_2^{(k)}}
\newcommand{\ok}{\omega^{(k)}}
\newcommand{\no}{u_1}
\newcommand{\nt}{u_2}
\newcommand{\vo}{v_1}
\newcommand{\vt}{v_2}

\newcommand{\uo}{u_1}
\newcommand{\ut}{u_2}

\newcommand{\fo}{F_1}
\newcommand{\ft}{F_2}
\newcommand{\gt}{G_2}
\newcommand{\go}{G_1}

\newcommand{\we}{w^{(\varepsilon)}}
\newcommand{\weg}{\left(w^{(\varepsilon)}\right)}
\newcommand{\ome}{\omega^{(\varepsilon)}}
\newcommand{\re}{R^{(\varepsilon)}}
\newcommand{\roe}{R_1^{(\varepsilon)}}
\newcommand{\rte}{R_2^{(\varepsilon)}}
\newcommand{\rk}{R^{(k)}}
\newcommand{\rok}{R_1^{(k)}}
\newcommand{\rtk}{R_2^{(k)}}

\newcommand{\uoe}{u_1^{(\varepsilon)}}
\newcommand{\ute}{u_2^{(\varepsilon)}}
\newcommand{\uok}{u_1^{(k)}}
\newcommand{\utk}{u_2^{(k)}}
\newcommand{\bk}{B_k(0)}
\newcommand{\pt}{\partial_t}

\title[a  model governing the motion of two cell populations ] 
{Global existence theorem for a  model governing the motion of two cell populations
}

\author[Brock C. Price and Xiangsheng Xu]{}

\email{bcp193@msstate.edu}
\email{xxu@math.msstate.edu}

\begin{document}
	\maketitle
	
	\centerline{\scshape Brock C. Price and Xiangsheng Xu}
	\medskip
	{\footnotesize
		\centerline{Department of Mathematics \& Statistics}
		\centerline{Mississippi State University}
		\centerline{ Mississippi State, MS 39762, USA}
	} 

	\bigskip

	\begin{abstract}This article is concerned with the existence of a weak solution to the initial boundary problem for a cross-diffusion system which arises in the study of two cell population growth. The mathematical challenge is due to the fact that the coefficient matrix is non-symmetric and degenerate in the sense that its determinant is $0$. Existence assertion is established by exploring the fact that the total population density satisfies a porous media equation. 
	
	 	\end{abstract}
	\section{Introduction}
Let $\Omega$ be a bounded domain in $\mathbb{R}^d$ with Lipschitz boundary
$\po$ and $T$ any positive number. We consider the initial boundary value problem 
\begin{eqnarray}
\partial_t\no-\mu\mdiv\left(\no\nabla w^\gamma\right)&=&R_1\ \ \ \mbox{in $\ot\equiv\Omega\times(0,T)$,}\label{euo}\\
\partial_t\nt-\nu\mdiv\left(\nt\nabla w^\gamma\right)&=&R_2\ \ \ \mbox{in $\ot$,}\label{eut}\\
\no\nabla w^\gamma\cdot\mathbf{n}&=&0\ \ \ \mbox{on $\Sigma_T\equiv\po\times(0,T) $,}\label{uob}\\
\nt\nabla w^\gamma\cdot\mathbf{n}&=&0\ \ \ \mbox{on $\Sigma_T$,}\label{utb}\\
(\uo(x,0),\ut(x,0))&=&(u^{(0)}_1(x), u^{(0)}_2(x))\ \ \mbox{on $\Omega$, }\label{uicon}
\end{eqnarray}
where $\mathbf{n}$ is the unit outward normal to $\po$, 
\begin{eqnarray}
w&=&\uo+\ut,\label{wdef}\\
R_1&=&\no\fo(w)+\nt\go(w),\ \ \mbox{and}\label{r1}\\
R_2&=&\no\ft(w)+\nt\gt(w).\label{r2}
\end{eqnarray}
Assume:
\begin{enumerate}
	\item[\textup{(H1)}] $F_i, G_i, i=1,2,$ are all continuous functions with the properties
	\begin{eqnarray}
	F(w)&\equiv&\fo(w)+\ft(w)\leq 0\ \ \mbox{on $[w_p, \infty)$,}\label{fgc1}\\
	G(w)&\equiv&\go(w)+\gt(w)\leq 0 \ \ \mbox{on $[w_p, \infty)$,}\ \ \mbox{and}\ \ \label{fgc2}\\
	E(w)&\equiv&\min\{F_1(w),\ft(w),\go(w), G_2(w)\}\geq 0\ \ \mbox{on $ [0, w_p)$ for some $w_p>0$;}\nonumber
	\end{eqnarray}
	\item[\textup{(H2)}] $\mu,\nu\in (0,\infty),\ \gamma> 1$;
	\item[\textup{(H3)}] $\uo^{(0)}(x)\geq 0,\ \ \ut^{(0)}(x)\geq 0$, and
	\begin{equation} 
	w(x,0)\leq w_p\ \ \mbox{on $\Omega$}.\label{winc}
	\end{equation}
\end{enumerate}
This problem can be used to describe the interaction between a population of dividing cells and a population of non-dividing cells (see \cite{GPS} and the references therein). In this case the function $w^\gamma$ represents the pressure.
The second terms on the left hand sides of the two equations \eqref{euo} and \eqref{eut} model the tendency of cells to move down
pressure gradients and rely on the definition of the cell velocity fields through Darcy's law \cite{BC}. The parameters $\mu, \nu$  stand for the mobility
(i.e. the quotient of permeability and viscosity) of dividing cells and non-dividing
cells, respectively. If $\mu\ne\nu$, then the two cell populations are characterized by
different mobilities.
Assumptions \eqref{fgc1} and \eqref{fgc2} mean that competition for space decreases the cell division rate
according to the local pressure. The parameter $w_p$ models the threshold
pressure above which dividing cells are entering a quiescent state (i.e. the so-called
homeostatic pressure) \cite{BD}.

The objective of this paper is to investigate approximation to the initial boundary value problem when $\mu$ and $\nu$ may be of different values and existence when $\mu=\nu$.

\begin{definition}We say that $(\uo, \ut)$ is a weak solution to \eqref{euo}-\eqref{uicon} if:
	\begin{enumerate}
		\item[\textup{(D1)}]$\uo,\ut$ are non-negative and bounded with
		\begin{equation*}
		\partial_t\uo,\ \ \partial_t\ut\in L^2(0,T; \left(W^{1,2}(\Omega)\right)^*),\ \ w^{\frac{\gamma+1}{2}}\in L^2(0,T; W^{1,2}(\Omega)),
		\end{equation*}
		where $w$ is given as in \eqref{wdef} and $\left(W^{1,2}(\Omega)\right)^*$ denotes the dual space of $W^{1,2}(\Omega)$;
		\item[\textup{(D2)}]there hold
		\begin{eqnarray*}
			-\ioT\uo\partial_t\varphi dxdt+\ioT\uo\nabla w^\gamma\cdot\nabla\varphi dxdt&=&\ioT R_1\varphi dxdt-\langle \uo(\cdot, T), \varphi(\cdot, T)\rangle\nonumber\\
			&&+\io \uo^{(0)}(x)\varphi(x, 0)dx,\\
			-\ioT\ut\partial_t\varphi dxdt+\ioT\ut\nabla w^\gamma\cdot\nabla\varphi dxdt&=&\ioT R_2\varphi dxdt-\langle \ut(\cdot, T), \varphi(\cdot, T)\rangle\nonumber\\
			&&+\io \ut^{(0)}(x)\varphi(x, 0)dx
		\end{eqnarray*}
		for each smooth function $\varphi$, 
		where $\langle \cdot,\cdot\rangle$ denotes the duality pairing between $W^{1,2}(\Omega)$ and $\left(W^{1,2}(\Omega)\right)^*$.
	\end{enumerate}
\end{definition}
To see that the two equations in (D2) make sense, we can conclude from (D1) that
$\uo,\ut\in C([0, T]; \left(W^{1,2}(\Omega)\right)^*)$. Since $ w$ is bounded and
$\gamma\geq \frac{\gamma+1}{2}$, we also have $w^\gamma\in  L^2(0,T; W^{1,2}(\Omega))$.
\begin{theorem}\label{mth} Let (H1)-(H3) be satisfied. Assume:
	\begin{enumerate}
		\item[\textup{(H4)}] $\uo^{(0)}, \ut^{(0)}\in W^{1,2}(\Omega)$; 
		\item[\textup{(H5)}] $\mu=\nu$.
	\end{enumerate}
	Then there is a weak solution to \eqref{euo}-\eqref{uicon}.	
\end{theorem}
In general, the rigorous mathematical analysis of nonlinear differential equations depends primarily upon deriving estimates, but typically also upon using these estimates to justify limiting procedures of various sorts. The two issues are closely related.  Our system here is a cross-diffusion one, and  mathematical analysis of systems of this type
has attracted a lot of attention recently. One approach (see, e.g., \cite{CDJ, CJ}) is to seek a possibly convex function $\psi$ on $\mathbb{R}^2$ so that $t \rightarrow
\io \psi(\uo(x,t), \ut(x,t))dx$ is a Lyapunov functional along the solutions to \eqref{euo}-\eqref{utb}. Unfortunately, this so-call entropy method cannot lead to
an existence assertion here. To see this,
we calculate
\begin{eqnarray}
\frac{d}{dt}\io\psi(\uo,\ut)dx&=&\io\left(\psi_{\uo}\partial_t\uo+\psi_{\ut}\partial_t\ut\right)dx\nonumber\\
&=&-\mu\gamma\io\uo w^{\gamma-1}\nabla w\cdot\left(\psi_{\uo\uo}\nabla\uo+\psi_{\uo\ut}\nabla\ut\right)dx\nonumber\\
&&-\nu\gamma\io\ut w^{\gamma-1}\nabla w\cdot\left(\psi_{\ut\uo}\nabla\uo+\psi_{\ut\ut}\nabla\ut\right)dx\nonumber\\
&&+\io\left(R_1\psi_{\uo}+R_2\psi_{\ut}\right)dx\nonumber\\
&=&-\gamma\io w^{\gamma-1}\left(a|\nabla\uo|^2+(a+b)\nabla\uo\cdot\nabla\ut+b|\nabla\ut|^2\right)dx\nonumber\\
&&+\io\left(R_1\psi_{\uo}+R_2\psi_{\ut}\right)dx,\label{ipsi}
\end{eqnarray}
where
\begin{eqnarray*}
	a&=&\mu\uo\psi_{\uo\uo}+\nu\ut\psi_{\ut\uo},\\
	b&=&\mu\uo\psi_{\ut\uo}+\nu\ut\psi_{\ut\ut}.
\end{eqnarray*}
\begin{proposition}\label{ABC}The quadratic term $A|\xi|^2+B\xi\cdot\eta+C|\eta|^2\geq 0$ for all $\xi,\eta\in \mathbb{R}^d$ if and only if 
	\begin{equation}\label{abc}
	A\geq 0, \ \ C\geq 0,\ \ \mbox{and}\ \  B^2\leq 4AC.
	\end{equation}
\end{proposition}
\begin{proof} If $A|\xi|^2+B\xi\cdot\eta+C|\eta|^2\geq 0$ for all $\xi,\eta\in \mathbb{R}^d$ then we must have 
	\begin{equation*}
	A\geq 0\ \ \mbox{and}\ \  C\geq 0.
	\end{equation*}
	If $A=0$, then we must have $B=0$. Otherwise, we could always choose $\xi, \eta$ so that
	\begin{equation*}
	(B\xi+C\eta)\cdot\eta<0.
	\end{equation*}  
	If $A>0$, then 	
	\begin{eqnarray}
	A|\xi|^2+B\xi\cdot\eta+C|\eta|^2&=&\left(\sqrt{A}\xi+\frac{B}{2\sqrt{A}}\eta\right)^2+\left(C-\frac{B^2}{4A}\right)|\eta|^2\geq 0.\label{abc1}
	\end{eqnarray}
	Taking any $\eta$ with $|\eta|=1$ and $\xi=-\frac{B}{2A}\eta$, we obtain \eqref{abc}.
	The converse is an easy consequence of \eqref{abc1}.
\end{proof}
Thus to ensure
\begin{equation*}
a|\nabla\uo|^2+(a+b)\nabla\uo\cdot\nabla\ut+b|\nabla\ut|^2\geq 0,
\end{equation*}
we must choose $\psi $ so that
\begin{equation*}
a\geq 0\ \ \mbox{and} \ \ a=b.
\end{equation*}
Then \eqref{ipsi} reduces to
\begin{eqnarray}
\lefteqn{\frac{d}{dt}\io\psi(\uo,\ut)dx+\gamma\io w^{\gamma-1}a|\nabla w|^2dx}\nonumber\\
&=&\io\left[(\uo\fo(w)+\ut\go(w))\psi_{\uo}+(\uo\ft(w)+\ut\gt(w))\psi_{\ut}\right]dx,\label{west}
\end{eqnarray}
which can only give us an estimate on the gradient of the sum of $\uo$ and $\ut$. This is not very surprising because our system is degenerate in the sense that the coefficient matrix 
\begin{equation*}
A=\gamma w^{\gamma-1}\left(\begin{array}{cc}
\mu\uo &\mu\uo\\
\nu\ut &\nu\ut
\end{array}\right)
\end{equation*}
has determinant $0$. As we shall see,  \eqref{west} is an important equation to us. But this alone is not enough for an existence assertion. To gain more information, we are forced to assume $\mu=\nu$. Under this assumption, the total density $w$ satisfies a porous media equation, and we wish to take advantage of this fact. To be more specific,
we employ a so-called weak convergence method\cite{E}. That is, construct an approximation and then pass to the limit.
The central issue is how to take the limit in the product
\begin{equation*}
u_i^{(\varepsilon)}\nabla \weg^{\gamma}, \ \ i=1,2,
\end{equation*}
in our approximate problems. As we indicated earlier, the hyperbolic nature of our system and the fact that $\mdet(A)=0$ prevent us from obtaining any meaningful estimates for the sequence $(\nabla\uoe, \nabla\ute)$. The idea in \cite{GPS} in the case $\Omega=\mathbb{R}^d$ was to 
prove the precompactness
of $\{\weg^\gamma\}$ in $L^2(0,T; W^{1,2}(\mathbb{R}^d))$. This required a rather sophisticated analysis. 
Indeed, the authors of \cite{GPS} achieved their goal by developing an extension of the Aronson-Benilan regularizing effect for porous media equations which provided estimates for the Laplacian of the pressure term $w^\gamma$.
In our case
we obtain the precompactness of $\{\weg^{\gamma+1}\}$ in $L^2(0,T; W^{1,2}(\Omega))$ and show that this is enough to justify passing to the limit. Our proof seems to be more direct and also simpler and requires
weaker assumptions. For example, we do not impose any assumptions on the second order partial derives of the initial data as did in \cite{GPS}. Moreover,  condition (7) in \cite{GPS}, which imposes restrictions on space dimensions and the growth of $|F(w) -G(w)|$ near $0$,  has also been removed.

If the initial data $u^{(1)}_0(x), u^{(2)}_0(x)$ have disjoint supports, a result in \cite{CFSS} indicates that they can remain disjoint for all $t>0$ at least in the case $d=1$. That is to say, the two cell populations are segregated. Our assumption \eqref{winc} does not exclude this possibility here.  We also refer the reader to \cite{LLP} and the references therein for numerical results that deal with how the mobilities change the morphology of the interfaces between the two cell populations and analytical study of traveling wave solutions with composite shapes and discontinuities for cell-density models
of avascular tumor growth.

This paper is organized as follows. In Section 2 we fabricate an approximation scheme for \eqref{euo}-\eqref{uicon} and prove an existence assertion for the approximate problems. Here we allow the possibility that $\mu\ne\nu$. In Section 3 we prove Theorem \ref{mth}. Then we give a brief indication on how to extend Theorem \ref{mth} to the case $\Omega=\RN$. Since all these are done under the assumption $\mu=\nu$, the case where  $\mu\ne\nu$ remains open.

\section{ the approximate problems }
In this section we design an approximation scheme for \eqref{euo}-\eqref{uicon} from a totally different perspective than
the one in \cite{GPS},
and then prove the existence of a solution to the approximate problems. Before we begin, we will need the following two classical lemmas.  
\begin{lemma}\label{comt}Let $f$ be a  non-negative function on $\ot$ and $\alpha> 0$. Assume that
	\begin{enumerate}
		\item[\textup{(C1)}] 
		$f^\alpha\in L^2(0,T;W^{1,2}(\Omega))$;
		\item[\textup{(C2)}]
		$\partial_tf\in L^2(0,T;\left(W^{1,2}(\Omega)\right)^*)$.
	\end{enumerate}
	Then 
	the function $t\rightarrow \io f^{\alpha+1}(x,t)dx$ is absolutely continuous on $[0,T]$ and
	\begin{equation}\label{lm}
	\frac{d}{dt}\io f^{\alpha+1}dx=(\alpha+1)\left\langle\partial_tf,f^\alpha\right\rangle.
	\end{equation}
	
\end{lemma}
If $\alpha=1$, this lemma is a special case of the well known Lions-Magenes lemma.
Formula \eqref{lm} is trivial if $f$ is smooth. The general case can be established by suitable approximation. We shall omit the details.
\begin{lemma}[Lions-Aubin]\label{la}
	Let $X_0, X$ and $X_1$ be three Banach spaces with $X_0 \subseteq X \subseteq X_1$. Suppose that $X_0$ is compactly embedded in $X$ and that $X$ is continuously embedded in $X_1$. For $1 \leq p, q \leq \infty$, let
	\begin{equation*}
	W=\{u\in L^{p}([0,T];X_{0}): \partial_t u\in L^{q}([0,T];X_{1})\}.
	\end{equation*}
	Then:
	\begin{enumerate}
		\item[\textup{(i)}] If $p  < \infty$, then the embedding of W into $L^p([0, T]; X)$ is compact.
		\item[\textup{(ii)}] If $p  = \infty$ and $q  >  1$, then the embedding of W into $C([0, T]; X)$ is compact. 
	\end{enumerate}
\end{lemma}
The proof of this lemma can be found in \cite{S}.

Our approximate problem is:
\begin{eqnarray}
\partial_tw-\gamma\mdiv\left[(\mu\uo+\nu\ut) w^{\gamma-1}\nabla w\right]-\varepsilon\Delta w&=&R_1+R_2\equiv R\ \ \mbox{in $\ot$,}\label{aew}\\
\partial_t\uo-\gamma\mu\mdiv\left(\uo  w^{\gamma-1}\nabla w\right)-\varepsilon\Delta \uo&=& R_1\ \mbox{in $\ot$,}\label{aeuo}\\
\partial_t\ut-\gamma\nu\mdiv\left(\ut  w^{\gamma-1}\nabla w\right)-\varepsilon\Delta \ut&=& R_2\ \mbox{in $\ot$,}\label{aeut}\\
\nabla w\cdot\mathbf{n}=\nabla \uo\cdot\mathbf{n}&=&\nabla \ut\cdot\mathbf{n}=0\ \ \mbox{on $\Sigma_T$,}\label{autb}\\
(w,\uo,\ut)\mid_{t=0}&=&(w^{(0)}(x),\uo^{(0)}(x),\ut^{(0)}(x) )\nonumber\\  &&\mbox{on $\Omega$,}\label{auicon}
\end{eqnarray}
where $\varepsilon>0$ and
\begin{equation*}
w^{(0)}(x)= \uo^{(0)}(x)+\ut^{(0)}(x).
\end{equation*}
Before we state our existence theorem, we let
\begin{equation*}
W(0,T)=\left\{\omega\in L^2(0,T; W^{1,2}(\Omega)): \partial_t \omega\in L^2\left(0,T; \left(W^{1,2}(\Omega)\right)^*\right)\right\}.
\end{equation*}
We mention in passing that it is not difficult to derive from Lemma \ref{lm} 
that $W(0,T)$ is contained in $C([0,T]; L^2(\Omega))$.
\begin{theorem}\label{athm} Let (H1)-(H3) be satisfied. Then there exists a triplet $(\uo, \ut, w)$ in the function space $\left(W(0,T)\right)^3$ such that
	\begin{enumerate}
		\item[\textup{(1)}] $\uo\geq 0,\ut\geq 0$, and $ w=\uo+\ut$ with $w\leq w_p$;
		\item[\textup{(2)}]	Equations \eqref{aew}-\eqref{auicon} are all satisfied in the usual weak sense.
	\end{enumerate}
\end{theorem}
This theorem 
will be established via the Leray-Schauder fixed point theorem (\cite{GT}, p.280). For this purpose,
we introduce the function
\begin{equation}\label{thdf}
\theta_p(s)=\left\{\begin{array}{cc}
0&\mbox{if $s\leq 0$,}\\
s&\mbox{if $0<s<w_p$,}\\
w_p&\mbox{if $s\geq w_p$,}
\end{array}\right.
\end{equation}
where $w_p$ is given as in (H1).
We define an operator $\mathbb{M}$ from $\left(L^2(\ot)\right)^3$ into itself as follows: Let $(\omega,v_1,v_2)\in \left(L^2(\ot)\right)^3 $. We first consider the initial boundary value problem
\begin{eqnarray}
\partial_tw&=&\gamma\mdiv\left[(\mu\theta_p(v_1)+\nu\theta_p(v_2))(\theta_p(v_1)+\theta_p(v_2))^{\gamma-1}\nabla w\right]+\varepsilon\Delta w\nonumber\\
&&+\theta_p(v_1) F(\theta_p(\omega))+\theta_p(v_2) G(\theta_p(\omega))\ \mbox{in $\ot$,}\label{aew1}\\
\nabla w\cdot\mathbf{n}&=&0\ \ \mbox{on $\Sigma_T$,}\\
w(x,0)&=&w^{(0)}(x)\ \ \mbox{on $\Omega$. }\label{aewi}
\end{eqnarray}
For given $(\omega,v_1,v_2)$ the equation \eqref{aew1} is linear and uniformly parabolic in $w$. Thus we can conclude from the classical result (\cite{LSU}, Chap. III) that there is a unique weak solution $w$ to \eqref{aew1}-\eqref{aewi} in the space $W(0,T)$.
Use the function $w$ so obtained to form the following two initial boundary problems 
\begin{eqnarray}
\partial_t\uo-\varepsilon\Delta \uo&=&\gamma\mu\mdiv\left[\theta_p(v_1)(\theta_p(v_1)+\theta_p(v_2))^{\gamma-1}\nabla w\right]\nonumber\\
&&+\theta_p(v_1)\fo(\theta_p(\omega))+\theta_p(v_2)\go(\theta_p(\omega))\ \mbox{in $\ot$,}\label{aeuo1}\\
\nabla \uo\cdot\mathbf{n}&=&0\ \ \mbox{on $\Sigma_T$,}\\
\uo(x,0)&=&\uo^{(0)}(x) \ \mbox{on $\Omega$, }\\
\partial_t\ut-\varepsilon\Delta \ut&=&\gamma\nu\mdiv\left[\theta_p(v_2)(\theta_p(v_1)+\theta_p(v_2))^{\gamma-1}\nabla w\right]\nonumber\\
&&+\theta_p(v_1)\ft(\theta_p(\omega))+\theta_p(v_2)\gt(\theta_p(\omega))\ \mbox{in $\ot$,}\label{aeut1}\\
\nabla \ut\cdot\mathbf{n}&=&0\ \ \mbox{on $\Sigma_T$,}\label{autb1}\\
\ut(x,0)&=&\ut^{(0)}(x) \ \ \mbox{on $\Omega$. }\label{auicon1 }
\end{eqnarray}
Each of the two problems here has a unique solution in $W(0,T)$. 
We define $(w,\uo,\ut)=\mathbb{M}(\omega,v_1,v_2)$.
Evidently, $\mathbb{M}$ is well-defined.
\begin{lemma} For each fixed $\varepsilon>0$,
	the operator $\mathbb{M}$ is compact, i.e., $\mathbb{M}$ is continuous and maps bounded sets into precompact ones.
\end{lemma}
\begin{proof} It is not difficult for us to see that the range of $\mathbb{M}$ is a bounded set in $W(0,T)$. Thus we can conclude from Lemma \ref{la} that  $\mathbb{M}$  maps bounded sets into precompact ones. The continuity of $\mathbb{M}$ is based upon the observation that
	if any subsequence of a sequence has a further convergent subsequence and all its convergent subsequences have the same limit then the whole sequence converges. Suppose
	\begin{equation}\label{ra}
	(\vok,\vtk,\ok)\rightarrow (\vo,\vt,\omega)\ \ \mbox{strongly in $\left(L^2(\ot)\right)^3$.}
	\end{equation}
	Then for any bounded continuous function $H$ on $\mathbb{R}$ we have
	\begin{eqnarray}
	(H(\vok),H(\vtk),H(\ok))&\rightarrow& (H(\vo),H(\vt),H(\omega))\nonumber\\
	&& \mbox{strongly in $\left(L^p(\ot)\right)^3$ for each $p\geq 1$.}\label{ra1}
	\end{eqnarray}
	To see this, we can conclude from \eqref{ra} that there is a subsequence $\{\vo^{(k_j)}\}$ of $\{\vok\}$ such that
	\begin{equation*}
	\vo^{(k_j)}\rightarrow \vo\ \ \mbox{a.e. on $\ot$.}
	\end{equation*}
	Subsequently, 
	\begin{equation*}
	H(\vo^{(k_j)})\rightarrow H(\vo)\ \ \mbox{a.e. on $\ot$.}
	\end{equation*}
	This combined with Egoroff's theorem implies that
	\begin{equation*}
	H(\vo^{(k_j)})\rightarrow H(\vo)\ \ \mbox{strongly in $\left(L^p(\ot)\right)^3$ for each $p\geq 1$.}
	\end{equation*}
	That is, any subsequence of $\{	H(\vok)\}$ has a further subsequence which converges to $H(\vo)$. Thus the whole sequence converges to $H(\vo)$. Hence \eqref{ra1} follows.
	
	We have
	\begin{eqnarray*}
		\theta_p(v_i^{(k)})&\rightarrow&\theta_p(v_i),\\  F_i(\theta_p(\ok))&\rightarrow& F_i(\theta_p(\omega)),\\  G_i(\theta_p(\ok))&\rightarrow& G_i(\theta_p(\omega))\nonumber\\
		&&  \mbox{stronly in $\left(L^p(\ot)\right)^3$ for each $p\geq 1$ and $i=1,2$.}
	\end{eqnarray*}
	Set
	\begin{equation*}
	(\uok,\utk,\wk)=\mathbb{M}(\vok,\vtk,\ok).
	\end{equation*}
	That is,
	\begin{eqnarray}
	\partial_t\wk&=&\gamma\mdiv\left[(\mu\theta_p(\vok)+\nu\theta_p(\vtk))\left(\theta_p(\vok)+\theta_p(\vtk)\right)^{\gamma-1}\nabla \wk\right]+\varepsilon\Delta \wk\nonumber\\
	&&+\theta_p(\vok) F(\theta_p(\ok))+\theta_p(\vtk) G(\theta_p(\ok))\ \mbox{in $\ot$,}\label{ra2}\\
	\nabla \wk\cdot\mathbf{n}&=&0\ \ \mbox{on $\Sigma_T$,}\\
	\wk(x,0)&=& w^{(0)}(x)\ \ \mbox{on $\Omega$, }\\
	\partial_t\uok&=&\gamma\mu\mdiv\left[\theta_p(\vok)\left(\theta_p(\vok)+\theta_p(\vok)\right)^{\gamma-1}\nabla \wk\right]+\varepsilon\Delta \uok\nonumber\\
	&&+\theta_p(\vok)\fo(\theta_p(\ok))+\theta_p(\vtk)\go(\theta_p(\ok))\ \mbox{in $\ot$,}\label{rauo2}\\
	\nabla \uok\cdot\mathbf{n}&=&0\ \ \mbox{on $\Sigma_T$,}\\
	\uok(x,0)&=&\uo^{(0)}(x) \ \mbox{on $\Omega$, }\\
	\partial_t\utk&=&\gamma\nu\mdiv\left[\theta_p(\vtk)\left(\theta_p(\vok)+\theta_p(\vtk)\right)^{\gamma-1}\nabla \wk\right]+\varepsilon\Delta \utk\nonumber\\
	&&+\theta_p(\vok)\ft(\theta_p(\ok))+\theta_p(\vtk)\gt(\theta_p(\ok))\ \mbox{in $\ot$,}\label{raeut2}\\
	\nabla \utk\cdot\mathbf{n}&=&0\ \ \mbox{on $\Sigma_T$,}\label{rautb2}\\
	\utk(x,0)&=&\ut^{(0)}(x) \ \ \mbox{on $\Omega$. }\label{rauicon2 }
	\end{eqnarray}
	Using $\wk$ as a test function in \eqref{ra2}, we can easily show that
	$\{\wk\}$ is bounded in $W(0,T)$. This together with \eqref{rauo2} and \eqref{raeut2} implies that $\{\uok\}$ and $\{\utk\}$ are also bounded in $W(0,T)$. By Lemma \ref{la}, we can extract a subsequence of $\{(\uok,\utk,\wk)\}$, still denoted by $\{(\uok,\utk,\wk)\}$, such that
	\begin{equation*}
	(\uok,\utk,\wk)\rightarrow (\uo,\ut, w)\ \ \mbox{stronly in $L^2(\ot)$.}
	\end{equation*}
	Furthermore,
	\begin{eqnarray*}
		\nabla\uok&\rightarrow&\nabla\uo\ \ \ \mbox{weakly in $L^2(0,T, \left(L^2(\Omega)\right)^d)$},\\
		\nabla\utk&\rightarrow&\nabla\ut\ \ \ \mbox{weakly in $L^2(0,T, \left(L^2(\Omega)\right)^d)$},\ \ \mbox{and}\\
		\nabla\wk&\rightarrow&\nabla w\ \ \ \mbox{weakly in $L^2(0,T, \left(L^2(\Omega)\right)^d)$}.
	\end{eqnarray*}
	Thus we can pass to the limit in \eqref{ra2}-\eqref{rauicon2 } to derive
	\begin{eqnarray}
	\partial_tw&=&\gamma\mdiv\left[(\mu\theta_p(\vo)+\nu\theta_p(\vt))\left(\theta_p(\vo)+\theta_p(\vt)\right)^{\gamma-1}\nabla w\right]+\varepsilon\Delta w\nonumber\\
	&&+\theta_p(\vo) F(\theta_p(\omega))+\theta_p(\vt) G(\theta_p(\omega))\ \mbox{in $\ot$,}\label{ra3}\\
	\nabla w\cdot\mathbf{n}&=&0\ \ \mbox{on $\Sigma_T$,}\\
	w(x,0)&=& w^{(0)}(x)\ \ \mbox{on $\Omega$, }\label{ra33}\\
	\partial_t\uo&=&\gamma\mu\mdiv\left[\theta_p(\vo)\left(\theta_p(\vo)+\theta_p(\vo)\right)^{\gamma-1}\nabla w\right]+\varepsilon\Delta \uo\nonumber\\
	&&+\theta_p(\vo)\fo(\theta_p(\omega))+\theta_p(\vt)\go(\theta_p(\omega))\ \mbox{in $\ot$,}\label{rauo3}\\
	\nabla \uo\cdot\mathbf{n}&=&0\ \ \mbox{on $\Sigma_T$,}\\
	\uo(x,0)&=&\uo^{(0)}(x) \ \mbox{on $\Omega$, }\\
	\partial_t\ut&=&\gamma\nu\mdiv\left[\theta_p(\vt)\left(\theta_p(\vo)+\theta_p(\vt)\right)^{\gamma-1}\nabla w\right]+\varepsilon\Delta \ut\nonumber\\
	&&+\theta_p(\vo)\ft(\theta_p(\omega))+\theta_p(\vt)\gt(\theta_p(\omega))\ \mbox{in $\ot$,}\label{raeut3}\\
	\nabla \ut\cdot\mathbf{n}&=&0\ \ \mbox{on $\Sigma_T$,}\label{rautb3}\\
	\ut(x,0)&=&\ut^{(0)}(x) \ \ \mbox{on $\Omega$. }\label{rauicon3 }
	\end{eqnarray}
	The solution $w$ to \eqref{ra3}-\eqref{ra33} is unique, and $\uo$ and $\ut$ are uniquely determined by  $w$. That is, there is only one solution to \eqref{ra3}-\eqref{rauicon3 }. This means that any subsequence of $(\uok,\utk,\wk)$ has the same limit $\mathbb{M}(\vo,\vt,\omega)$. Thus the whole sequence also converges to it. This completes the proof. 
\end{proof}
\begin{lemma}
	There is a positive number $c$ such that
	\begin{equation}\label{lse}
	\|(w,\uo,\ut)\|_{\left(L^2(\ot)\right)^3}\leq c
	\end{equation}
	for all $(w,\uo,\ut)\in \left(L^2(\ot)\right)^3$ and $\sigma\in (0,1)$ satisfying
	\begin{equation*}
	(w,\uo,\ut)	=\sigma\mathbb{M}(w,\uo,\ut).
	\end{equation*}
\end{lemma}
\begin{proof} It is easy to see that the above equation is equivalent to the following problem
	\begin{eqnarray}
	\partial_tw&=&\gamma\mdiv\left[(\mu\theta_p(u_1)+\nu\theta_p(u_2))(\theta_p(u_1)+\theta_p(u_2))^{\gamma-1}\nabla w\right]+\varepsilon\Delta w\nonumber\\
	&&+\sigma\theta_p(u_1) F(\theta_p(w))+\sigma\theta_p(u_2) G(\theta_p(w))\ \mbox{in $\ot$,}\label{aew2}\\
	\nabla w\cdot\mathbf{n}&=&0\ \ \mbox{on $\Sigma_T$,}\\
	w(x,0)&=&\sigma w^{(0)}(x)\ \ \mbox{on $\Omega$, }\\
	\partial_t\uo&=&\gamma\mu\mdiv\left[\theta_p(u_1)(\theta_p(u_1)+\theta_p(u_2))^{\gamma-1}\nabla w\right]+\varepsilon\Delta \uo\nonumber\\
	&&+\sigma\theta_p(u_1)\fo(\theta_p(w))+\sigma\theta_p(u_2)\go(\theta_p(w))\ \mbox{in $\ot$,}\label{aeuo2}\\
	\nabla \uo\cdot\mathbf{n}&=&0\ \ \mbox{on $\Sigma_T$,}\\
	\uo(x,0)&=&\sigma\uo^{(0)}(x) \ \mbox{on $\Omega$, }\\
	\partial_t\ut&=&\gamma\nu\mdiv\left[\theta_p(u_2)(\theta_p(u_1)+\theta_p(u_2))^{\gamma-1}\nabla w\right]+\varepsilon\Delta \ut\nonumber\\
	&&+\sigma\theta_p(u_1)\ft(\theta_p(w))+\theta_p(u_2)\gt(\theta_p(w))\ \mbox{in $\ot$,}\label{aeut2}\\
	\nabla \ut\cdot\mathbf{n}&=&0\ \ \mbox{on $\Sigma_T$,}\label{autb2}\\
	\ut(x,0)&=&\sigma\ut^{(0)}(x) \ \ \mbox{on $\Omega$. }\label{auicon2 }
	\end{eqnarray}
	Add \eqref{aeut2} to \eqref{aeuo2} and subtract the resulting equation from \eqref{aew2} to derive
	\begin{equation*}
	\partial_t(w-(\uo+\ut))-\varepsilon\Delta(w-(\uo+\ut))=0\ \ \mbox{in $\ot$.}
	\end{equation*}
	Recall the initial boundary conditions for $(w-(\uo+\ut))$ to deduce
	\begin{equation}\label{wot}
	w=\uo+\ut.
	\end{equation}
	Use $(w-w_p)^+$ as a test function \eqref{aeuo2} to obtain
	\begin{eqnarray*}
		\lefteqn{\frac{1}{2}\frac{d}{dt}\io\left[(w-w_p)^+\right]^2dx}\nonumber\\
		&&+\gamma\io\left[(\mu\theta_p(u_1)+\nu\theta_p(u_2))(\theta_p(u_1)+\theta_p(u_2))^{\gamma-1}+\varepsilon\right]|\nabla(w-w_p)^+|^2dx\nonumber\\
		&&=\io\left(\sigma\theta_p(u_1) F(\theta_p(w))+\sigma\theta_p(u_2) G(\theta_p(w))\right)(w-w_p)^+dx= 0.
	\end{eqnarray*}
	The last step is due to the definition of $\theta_p$ \eqref{thdf} and (H1).	Integrate with respect to $t$ to yield
	\begin{equation}\label{wub}
	w\leq w_p.
	\end{equation}
	Thus we can replace $\theta_p(w)$ in the preceding equations by $w$.
	Note that 
	\begin{equation*}
	\theta_p(u_1)\uo^-=0\ \ \mbox{and}\ \ \theta_p(u_2) \go(w)\geq 0\ \ \mbox{because $w\leq w_p$.}
	\end{equation*}
	With this in mind, we use $\uo^-$ as a test function in \eqref{aeuo2} to derive
	\begin{eqnarray*}
		-\frac{1}{2}\frac{d}{dt}\io\left(\uo^-\right)^2dx-\varepsilon\io|\nabla\uo^-|^2dx=\sigma\io\theta_p(u_2) \go(w)\uo^-dx\geq 0,
	\end{eqnarray*}
	from whence follows
	\begin{equation*}
	\uo\geq 0.
	\end{equation*}
	By the same token, 
	\begin{equation*}
	\ut\geq 0.
	\end{equation*}
	In view of \eqref{thdf}, \eqref{wot}, and \eqref{wub}, we have
	\begin{equation*}
	\theta_p(u_1)=\uo,\ \ \theta_p(u_2)=\ut.
	\end{equation*}
	Use $w$ as a test function \eqref{aew2} to obtain
	\begin{eqnarray}
	\frac{1}{2}\frac{d}{dt}\io w^2dx+\gamma\io(\mu\uo+\nu\ut)w^{\gamma-1}|\nabla w|^2dx+\varepsilon\io|\nabla w|^2dx\leq M_0\io w^2dx,\label{awe2}
	\end{eqnarray}
	where
	\begin{equation}\label{m0}
	M_0=\max\{\max_{w\in [0,w_p]}F(w),\max_{w\in [0,w_p]}G(w) \}.
	\end{equation}
	Use Gronwall's inequality in \eqref{awe2} to obtain 
	\begin{eqnarray}
	\sup_{0\leq t\leq T}\io w^2dx+\ioT|\nabla w|^2dxdt\leq c.\label{awe1}
	\end{eqnarray}
	Similarly, we can prove that $\uo, \ut$ are bounded in $W(0,T)\subset L^2(\ot)$. This completes the proof.
\end{proof}

Theorem \ref{athm} is a consequence of the preceding two lemmas and the Leray-Schauder fixed point theorem.
\section{Proof of Theorem \ref{mth}}

The proof of Theorem \ref{mth} is divided into several lemmas.

For each $1\geq\varepsilon>0$ we denote by $(\we,\uoe, \ute)$ the solution to \eqref{aew}-\eqref{auicon} constructed earlier. Subsequently, we have
\begin{equation*}
\uoe\geq 0,\ \ \ute\geq 0, \ \ \we=\uoe+\ute.
\end{equation*}
Moreover,
\begin{equation*}
\we\leq w_p.
\end{equation*}
Set
\begin{eqnarray*}
	\re&=&\uoe F(\we)+\ute G(\we),\\
	\roe&=&\uoe\fo(\we)+\ute\go(\we),\\
	\rte&=&\uoe\ft(\we)+\ute\gt(\we).
\end{eqnarray*}
It follows that
\begin{equation*}
0\leq \re,\ \roe,\ \rte\leq w_pM_0,
\end{equation*}
where $M_0$ is given as in \eqref{m0}.
We can write \eqref{aew}-\eqref{auicon} in the form
\begin{eqnarray}
\partial_t\we-\mdiv\left[(\mu\uoe+\nu\ute)\nabla \weg^\gamma\right]-\varepsilon\Delta \we&=&\re\ \mbox{in $\ot$,}\label{aew4}\\
\partial_t\uoe-\mu\mdiv\left[\uoe\nabla  \weg^\gamma\right]-\varepsilon\Delta \uoe&=&\roe\ \mbox{in $\ot$,}\label{aeuo4}\\
\partial_t\ute-\nu\mdiv\left[\ute\nabla \weg^\gamma\right]-\varepsilon\Delta \ute&=&\rte\ \mbox{in $\ot$,}\label{aeut4}\\
\nabla \we\cdot\mathbf{n}=\nabla \uoe\cdot\mathbf{n}&=&\nabla \ute\cdot\mathbf{n}=0\ \ \mbox{on $\Sigma_T$,}\label{autb4}\\
(\we(x,0),\uoe(x,0),\ute(x,0))&=&(w^{(0)}(x),\uo^{(0)}(x),\ut^{(0)}(x) )\nonumber\\
&& \mbox{on $\Omega$. }\label{auicon4}
\end{eqnarray}
Our key estimate is the following
\begin{lemma}\label{l41}
	We have
	\begin{eqnarray*}
		\ioT\left|\nabla\left(\we\right)^{\frac{\gamma+1}{2}}\right|^2dxdt+\varepsilon\ioT\left(\left|\nabla\sqrt{\uoe}\right|^2+\left|\nabla\sqrt{\ute}\right|^2\right)dxdt\leq c.
	\end{eqnarray*}
\end{lemma}
\begin{proof}
	Pick $\tau>0$. Use $\frac{1}{\mu}\ln(\uoe+\tau)$ as a test function in \eqref{aeuo4} to derive
	\begin{eqnarray*}
		\lefteqn{\frac{1}{\mu}\frac{d}{dt}\io\left((\uoe+\tau)\ln(\uoe+\tau)-\uoe\right)dx+\io \frac{\uoe}{\uoe+\tau}\nabla  \weg^\gamma\nabla\uoe dx}\nonumber\\
		&&+\frac{\varepsilon}{\mu}\io\frac{1}{\uoe+\tau}|\nabla\uoe|^2\nonumber\\
		&=& \frac{1}{\mu}\io \left(\uoe\fo(\we)+\ute\go(\we)\right)\ln(\uoe+\tau)dx\nonumber\\
		&\leq&\frac{1}{\mu}\int_{\{\uoe+\tau\geq 1\}}\left(\uoe\fo(\we)+\ute\go(\we)\right)\ln(\uoe+\tau)dx\nonumber\\
		&\leq &\frac{M_0}{\mu}\io\we(\uoe+\tau)dx.
	\end{eqnarray*}
	Integrate, note that $\sup_{\ot}|\uoe\ln\uoe|\leq c$, and take $\tau\rightarrow 0$ to get
	\begin{equation*}
	\ioT\nabla\weg^{\gamma}\cdot\nabla\uoe dxdt+\frac{4\varepsilon}{\mu}\ioT\left|\nabla\sqrt{\uoe}\right|^2dxdt\leq c.
	\end{equation*}
	Similarly,
	\begin{equation*}
	\ioT\nabla\weg^{\gamma}\cdot\nabla\ute dxdt+\frac{4\varepsilon}{\mu}\ioT\left|\nabla\sqrt{\ute}\right|^2dxdt\leq c.
	\end{equation*}
	Add up the two preceding inequalities to obtain the desired result.
\end{proof}
\begin{lemma}\label{l32} The sequence $\{\we\}$ is precompact in $L^p(\ot)$ for each $p\geq 1$.
\end{lemma}
\begin{proof}
	Without loss of generality, we may assume
	\begin{equation}\label{wep}
	\we\geq\varepsilon. 
	\end{equation}This can be achieved easily by replacing $\uo^{(0)}$ with $\uo^{(0)}+\varepsilon$ in our approximate problems. Indeed,
	use $(\varepsilon-\we)^+$ as a test function in \eqref{aew4} to get
	\begin{eqnarray*}
		\lefteqn{-\frac{1}{2}\frac{d}{dt}\io\left[(\varepsilon-\we)^+\right]^2dx-\gamma\io(\mu\uoe+\nu\ute)\weg^{\gamma-1}|\nabla(\varepsilon-\we)^+|^2dx}\nonumber\\
		&&-\varepsilon\io|\nabla(\varepsilon-\we)^+|^2dx=\io\re(\varepsilon-\we)^+dx\geq 0.
	\end{eqnarray*}
	This gives \eqref{wep}.
	
	We derive from \eqref{aew4} that
	\begin{eqnarray}
	\partial_t\left(\we\right)^{\frac{\gamma+1}{2}}&=&\frac{\gamma+1}{2}\left(\we\right)^{\frac{\gamma+1}{2}-1}\partial_t\we\nonumber\\
	&=&\frac{\gamma+1}{2}\mdiv\left[(\mu\uoe+\nu\ute)\left(\we\right)^{\frac{\gamma+1}{2}-1}\nabla \weg^\gamma\right]\nonumber\\
	&&-\frac{\gamma+1}{2}(\mu\uoe+\nu\ute)\nabla\left(\we\right)^{\frac{\gamma+1}{2}-1}\cdot\nabla \weg^\gamma\nonumber\\
	&&+\frac{(\gamma+1)\varepsilon}{2}\mdiv\left[\left(\we\right)^{\frac{\gamma+1}{2}-1}\nabla\we\right]-\frac{(\gamma+1)\varepsilon}{2}\nabla\left(\we\right)^{\frac{\gamma+1}{2}-1}\cdot\nabla \we\nonumber\\ &&+\frac{\gamma+1}{2}\left(\we\right)^{\frac{\gamma+1}{2}-1}\re\nonumber\\
	&=&\gamma\mdiv\left[(\mu\uoe+\nu\ute)\weg^{\gamma-1}\nabla\left(\we\right)^{\frac{\gamma+1}{2}} \right]\nonumber\\
	&&-\frac{\gamma(\gamma-1)}{\gamma+1}(\mu\uoe+\nu\ute)\left(\we\right)^{\frac{\gamma+1}{2}-2}\left|\nabla\left(\we\right)^{\frac{\gamma+1}{2}}\right|^2\nonumber\\
	&&+\varepsilon\Delta\left(\we\right)^{\frac{\gamma+1}{2}}-(\gamma^2-1)\varepsilon\left(\we\right)^{\frac{\gamma+1}{2}-1}\left|\nabla \sqrt{\we}\right|^2\nonumber\\ &&+\frac{\gamma+1}{2}\left(\we\right)^{\frac{\gamma+1}{2}-1}\re.\label{wcom}
	\end{eqnarray}
	Remember that $\frac{\gamma+1}{2}-1>0$. It follows that
	\begin{equation*}
	(\mu\uoe+\nu\ute)\left(\we\right)^{\frac{\gamma+1}{2}-2}\leq \max\{\mu,\nu\}\left(\we\right)^{\frac{\gamma+1}{2}-1}\leq c.
	\end{equation*}
	We can conclude that the sequence $\{\partial_t\left(\we\right)^{\frac{\gamma+1}{2}}\}$ is bounded in $L^2(0, T; \left(W^{1,2}(\Omega)\right)^*)+L^1(\ot)	$. This together with Lemma \ref{l41} enables us to use (i) in Lemma \ref{la}, thereby obtaining the precompactness
	of $\{\left(\we\right)^{\frac{\gamma+1}{2}}\}$ in $L^2(\ot)$. Then the lemma follows from the boundedness of $\{\we\}$.

\end{proof}

We may extract a subsequence of $\{(\uoe,\ute,\we)\}$, still denoted by $\{(\uoe,\ute,\we)\}$, such that
\begin{eqnarray}
\uoe&\rightarrow&\uo \ \ \mbox{weak$^*$ in $L^\infty(\ot)$,}\\
\ute&\rightarrow&\ut \ \ \mbox{weak$^*$ in $L^\infty(\ot)$,}\\
\we&\rightarrow&w \ \ \mbox{a.e. in $\ot$ and strongly in $L^p(\ot)$ for each $p\geq 1$, and}\label{wc}\\
\left(\we\right)^{\frac{\gamma+1}{2}}&\rightarrow& w^{\frac{\gamma+1}{2}}\ \ \mbox{weakly in $L^2(0,T; W^{1,2}(\Omega))$.}
\end{eqnarray}
Since $\{\we\}$ is bounded, we also have
\begin{equation*}
\left(\we\right)^{p}\rightarrow w^{p}\ \ \mbox{weakly in $L^2(0,T; W^{1,2}(\Omega))$ for each $p\geq \frac{\gamma+1}{2}$.}
\end{equation*}
This combined with \eqref{aew4} implies
\begin{equation*}
\partial_t\we\rightarrow \partial_tw\ \ \mbox{weakly in $L^2(0,T; \left(W^{1,2}(\Omega)\right)^*)$.}
\end{equation*}
In view of \eqref{ra1}, we yield
\begin{eqnarray}
F_i(\we)&\rightarrow &F_i(w)\ \ \mbox{strongly in $L^p(\ot)$ for each $p\geq 1$}\ \ \mbox{and }\label{r1c}\\ 
G_i(\we)&\rightarrow& G_i(w)\ \ \mbox{strongly in $L^p(\ot)$ for each $p\geq 1$, $i=1,2$.}\label{r2c}
\end{eqnarray}
Subsequently,
\begin{eqnarray}
\re\rightarrow R,\ \ \roe\rightarrow R_1,\ \ \rte\rightarrow R_2\ \ \mbox{weak$^*$ in $L^\infty(\ot)$.}\label{rhl}
\end{eqnarray}

\begin{lemma}\label{l33}
	If $\mu=\nu$, then 
	\begin{equation*}
	\left(\we\right)^{\gamma+1}\rightarrow w^{\gamma+1}\ \ \mbox{strongly in $L^2(0,T;W^{1,2}(\Omega))$.}
	\end{equation*}
\end{lemma}
\begin{proof} In this case, we have
	\begin{equation}\label{ws1}
	(\mu\uoe+\nu\ute)\nabla \weg^\gamma=\mu\we \nabla(\we)^\gamma=\frac{\mu\gamma}{\gamma+1}\nabla\left(\we\right)^{\gamma+1}.
	\end{equation}
	Thus we can write \eqref{aew4} in the form
	\begin{equation}\label{aew44}
	\partial_t\we-\frac{\mu\gamma}{\gamma+1}\Delta\omega^{\varepsilon}
	=\re,
	\end{equation}
	where
	\begin{equation*}
	\omega^{\varepsilon}=\left(\we\right)^{\gamma+1}+\frac{\varepsilon(\gamma+1)}{\mu\gamma} \we.
	\end{equation*}
	We may assume that $\we$ is a classical solution to \eqref{aew44} because it can be viewed as the limit of a sequence of classical approximate solutions. Use $\partial_t	\ome$ as a test function in \eqref{aew44} to derive
	\begin{equation}\label{end1}
	\io 	\partial_t\we\partial_t	\ome dx+\frac{\mu\gamma}{\gamma+1}\io\nabla\omega^{\varepsilon}\cdot\nabla\partial_t	\ome dx=\io\re\partial_t	\ome dx
	\end{equation}
	We proceed to evaluate each integral in the above equation as follows:
	\begin{eqnarray*}
		\io 	\partial_t\we\partial_t	\ome dx&=&(\gamma+1)\io \left(\we\right)^{\gamma}\left(\partial_t\we\right)^2dx+\frac{\varepsilon(\gamma+1)}{\mu\gamma}\io\left(\partial_t\we\right)^2dx,\\
		\io\nabla\omega^{\varepsilon}\cdot\nabla\partial_t	\ome dx&=&\frac{1}{2}\frac{d}{dt}\io\left|\nabla\ome\right|^2dx,\\
		\io\re\partial_t	\ome dx&=&(\gamma+1)\io\re	\left(\we\right)^{\gamma}\partial_t\we dx+\frac{\varepsilon(\gamma+1)}{\mu\gamma}\io\re\partial_t\we dx\nonumber\\
		&\leq &\frac{\gamma+1}{2}\io \left(\we\right)^{\gamma}\left(\partial_t\we\right)^2dx+\frac{\gamma+1}{2}\io \left(\we\right)^{\gamma}\left(\re\right)^2dx\nonumber\\
		&&+\frac{\varepsilon(\gamma+1)}{2\mu\gamma}\io\left(\partial_t\we\right)^2dx++\frac{\varepsilon(\gamma+1)}{2\mu\gamma}\io\left(\re\right)^2dx
	\end{eqnarray*}
	Plug the preceding three results into \eqref{end1} and integrate to derive
	\begin{equation*}
	\ioT\left(\partial_t\weg^{\frac{\gamma+2}{2}}\right)^{2}dxdt+\varepsilon\ioT\left(\partial_t\we\right)^2dxdt+\sup_{0\leq t\leq T}\io\left|\nabla\ome\right|^2dx\leq c.
	\end{equation*}
	Since $\gamma+1>\frac{\gamma+2}{2}$, we also have that $\{\partial_t\left(\we\right)^{\gamma+1}\}$ is bounded in $L^2(\ot)$. By (ii) in Lemma \ref{la}, the sequence $\{\ome\}$ is precompact in $C([0, T], L^2(\Omega))$.
	It immediately follows from the boundedness of $\{\we\}$ that $\{\weg^{\gamma+1}\}$ is precompact in $C([0, T], L^p(\Omega))$ for each $p\geq 1$. Consequently,
	\begin{equation}\label{pc}
	\io\left(\we(x,t)\right)^qdx\rightarrow\io w^q(x,t)dx\ 
	\ \mbox{for each $t\in [0,T]$ and each $q\geq \gamma+1$.}
	\end{equation}

	Take $\varepsilon\rightarrow 0$ in \eqref{aew44} to obtain
	\begin{equation*}
	\partial_tw-\frac{\mu\gamma}{\gamma+1}\Delta w^{\gamma+1}= R.
	\end{equation*}	
	Subtract this equation from \eqref{aew44} and keep \eqref{ws1} in mind to get
	\begin{eqnarray}
	\partial_t(\we-w)-\frac{\mu\gamma}{\gamma+1}\Delta\left[\left(\we\right)^{\gamma+1}- w^{\gamma+1}\right]-\varepsilon\Delta \we
	&=&\re-R.\label{ws2}
	\end{eqnarray}
	Use $\left(\we\right)^{\gamma+1}- w^{\gamma+1}$ as a test function \eqref{ws2} to derive
	\begin{eqnarray}
	\lefteqn{\frac{\mu\gamma}{\gamma+1}\ioT\left|\nabla\left[\left(\we\right)^{\gamma+1}- w^{\gamma+1}\right]\right|^2dxdt}\nonumber\\
	&&+\varepsilon\ioT\nabla\we\cdot\nabla\left[\left(\we\right)^{\gamma+1}- w^{\gamma+1}\right]dxdt\nonumber\\
	&=& \ioT (\re-R)\left[\left(\we\right)^{\gamma+1}- w^{\gamma+1}\right]dxdt\nonumber\\
	&&-\int_0^T\left\langle\partial_t(\we-w),\left(\we\right)^{\gamma+1}- w^{\gamma+1}\right\rangle dt.\label{ws4}
	\end{eqnarray}	
	We will show that the last three terms in the above equation all go to $0$ as $\varepsilon\rightarrow 0$.
	It is easy to see from Lemma \ref{l41} that
	\begin{eqnarray}
	\lefteqn{\left|\varepsilon\ioT\nabla\we\cdot\nabla\left[\left(\we\right)^{\gamma+1}- w^{\gamma+1}\right]dxdt\right|}\nonumber\\
	&=&4\varepsilon\left|\ioT\sqrt{\we}\nabla\sqrt{\we}\cdot\left[\left(\we\right)^{\frac{\gamma+1}{2}}\nabla\left(\we\right)^{\frac{\gamma+1}{2}}- w^{\frac{\gamma+1}{2}}\nabla w^{\frac{\gamma+1}{2}}\right]dxdt\right|\nonumber\\
	&\leq &c\sqrt{\varepsilon}\rightarrow 0\ \ \mbox{as $\varepsilon\rightarrow 0$.}
	\end{eqnarray}
	By \eqref{wc} and \eqref{rhl}, we have
	\begin{equation*}
	\ioT (\re-R)\left[\left(\we\right)^{\gamma+1}- w^{\gamma+1}\right]dxdt\rightarrow 0\ \ \mbox{as $\varepsilon\rightarrow 0$.}
	\end{equation*}
	Finally,  we compute from Lemma \ref{lm} and \eqref{pc} that
	\begin{eqnarray*}
		\lefteqn{\int_0^T\left\langle\partial_t(\we-w),\left(\we\right)^{\gamma+1}- w^{\gamma+1}\right\rangle dt}\nonumber\\
		&=&\frac{1}{\gamma+2}\int_{0}^{T}\left[\frac{d}{dt}\io\left(\we\right)^{\gamma+2}dx+\frac{d}{dt}\io w ^{\gamma+2}dx\right]dt\nonumber\\
		&&-\int_0^T\left\langle\partial_t\we, w^{\gamma+1}\right\rangle dt-\int_0^T\left\langle\partial_t w, \weg^{\gamma+1}\right\rangle dt\nonumber\\
		&=&\frac{1}{\gamma+2}\left[\io\left(\we(x,T)\right)^{\gamma+2}dx+\io w^{\gamma+2}(x,T) dx\right]\nonumber\\
		&&-\frac{2}{\gamma+2}\io\left(w^{(0)}(x)\right)^{\gamma+2}dx-\int_0^T\left\langle\partial_t\we, w^{\gamma+1}\right\rangle dt-\int_0^T\left\langle\partial_t w, \weg^{\gamma+1}\right\rangle dt\nonumber\\
		&\rightarrow&\frac{2}{\gamma+2}\io w^{\gamma+2}(x,T) dx-\frac{2}{\gamma+2}\io\left(w^{(0)}(x)\right)^{\gamma+2}dx-2\int_0^T\left\langle\partial_tw, w^{\gamma+1}\right\rangle dt=0.
	\end{eqnarray*}
	This completes the proof.
\end{proof}
\begin{proof}[Proof of Theorem \ref{mth}]
	Equipped with this lemma, we can complete the proof of Theorem \ref{mth}.
	Keeping \eqref{wep} in mind, we can
	set
	\begin{equation*}
	\eta_1^{(\varepsilon)}=\frac{\uoe}{\we},\ \ \eta_2^{(\varepsilon)}=\frac{\ute}{\we}.
	\end{equation*}
	Suppose
	\begin{equation*}
	\eta_1^{(\varepsilon)}\rightarrow \eta_1,\ \ \eta_2^{(\varepsilon)}\rightarrow \eta_2\ \ \mbox{weak$^*$ in $L^\infty(\ot)$.}
	\end{equation*}
	We calculate
	\begin{eqnarray}
	\uoe\nabla\weg^{\gamma}&=&\eta_1^{(\varepsilon)}\we\nabla\weg^{\gamma}\nonumber\\
	&=&\frac{\gamma}{\gamma+1}\eta_1^{(\varepsilon)}\nabla\weg^{\gamma+1}\nonumber\\
	&\rightarrow &\frac{\gamma}{\gamma+1}\eta_1\nabla w^{\gamma+1}
	=\eta_1w\nabla w^{\gamma} \ \mbox{weakly in $\left(L^2(\ot)\right)^d$.}
	\end{eqnarray}
	We claim that
	\begin{equation}\label{ff}
	\eta_1w=\uo\ \ \mbox{a.e. on $\ot$}.
	\end{equation}
	To see this, for each $\delta>0$ we have
	\begin{equation*}
	\eta_1^{(\varepsilon)}(\we-\delta)^+\rightarrow \eta_1(w-\delta)^+\ \ \mbox{weak$^*$ in $L^\infty(\ot)$.}
	\end{equation*}
	Note that $\frac{(\we-\delta)^+}{\we}\leq 1$.
	Subsequently,
	\begin{equation*}
	\eta_1^{(\varepsilon)}(\we-\delta)^+=\uoe\frac{(\we-\delta)^+}{\we}\rightarrow \uo\frac{(w-\delta)^+}{w}\ \ \mbox{weak$^*$ in $L^\infty(\ot)$.}
	\end{equation*}
	We obtain
	\begin{equation*}
	\uo\frac{(w-\delta)^+}{w}=\eta_1(w-\delta)^+ \ \ \mbox{for each $\delta>0$}.
	\end{equation*}
	This implies that
	\begin{equation*}
	\uo=w\eta_1\ \ \mbox{on the set $\{w>0\}$.}
	\end{equation*}
	If $w=0$, then $\uo=0$, and we still have $\uo=w\eta_1$. This completes the proof of \eqref{ff}.
	Similarly, we can show
	\begin{equation*}
	\ute\nabla\weg^{\gamma}\rightarrow \ut\nabla w^\gamma\ \ \mbox{weakly in $\left(L^2(\ot)\right)^d$.}
	\end{equation*}
	We are ready to pass to the limit in \eqref{aeuo4} and \eqref{aeut4}, thereby finishing the proof of Theorem \ref{mth}.
\end{proof}

Finally, we remark that we can extend  Theorem \ref{mth} to the case considered in \cite{GPS}. For this purpose, we set $\mu=\nu=1$ and let $w, R, R_1, R_2 $ be given as before.
\begin{corollary}
	Let (H1), (H2), (H5) be satisfied,  and assume that (H3) holds for $\Omega=\RN$ and $\uo^{(0)},\ut^{(0)}\in W^{1,2}_{\textup{loc}}(\RN)$. Then there is a weak solution to the initial value problem
	\begin{eqnarray}
	\pt\uo-\mdiv\left(\uo\nabla w^\gamma\right)&=&R_1\ \ \mbox{in $\RN\times(0,T)$},\\
	\pt\ut-\mdiv\left(\ut\nabla w^\gamma\right)&=&R_2\ \ \mbox{in $\RN\times(0,T)$},\\
	(\uo,\ut)\mid_{t=0}&=&(\uo^{(0)}(x),\ut^{(0)}(x))\ \ \mbox{on $\RN$}
	\end{eqnarray}
	in the following sense:
	\begin{enumerate}
		\item[\textup{(C1)}]$\uo,\ut$ are non-negative and bounded with
		\begin{equation*}
		 w^{\frac{\gamma+1}{2}}\in L^2(0,T; W^{1,2}_{\textup{loc}}(\RN));
		\end{equation*}
		\item[\textup{(C2)}]there hold
		\begin{eqnarray*}
			-\int_{\RN\times(0,T)}\uo\partial_t\varphi dxdt+\int_{\RN\times(0,T)}\uo\nabla w^\gamma\cdot\nabla\varphi dxdt&=&\int_{\RN\times(0,T)} R_1\varphi dxdt
			+\int_{\RN} \uo^{(0)}(x)\varphi(x, 0)dx,\\
			-\int_{\RN\times(0,T)}\ut\partial_t\varphi dxdt+\int_{\RN\times(0,T)}\ut\nabla w^\gamma\cdot\nabla\varphi dxdt&=&\int_{\RN\times(0,T)} R_2\varphi dxdt
			+ \int_{\RN}\ut^{(0)}(x)\varphi(x, 0)dx
		\end{eqnarray*}
		for each smooth function $\varphi$ with compact support and $\varphi(x,T)=0$. 
	\end{enumerate}
\end{corollary}
We will give a brief outline of the proof. To this end,
we set
\begin{eqnarray*}
\bk&=&\{x\in\RN: |x|<k\},\ \ k=1,2,\cdots.
\end{eqnarray*}
We replace $\Omega$ in \eqref{euo}-\eqref{uicon} by $\bk$ and 
denote the resulting solution by $(\uok, \utk)$. That is, we have
\begin{eqnarray}
\partial_t\uok-\mdiv\left[\uok\nabla\wkg^\gamma\right]&=&\rok  \ \ \mbox{in $\bk\times(0,T)$},\label{inf1}\\
\partial_t\utk-\mdiv\left[\utk\nabla\wkg^\gamma\right]&=&\rtk  \ \ \mbox{in $\bk\times(0,T)$},\label{inf2}\\
\uok\nabla\wkg^\gamma\cdot\mathbf{n}&=&\utk\nabla\wkg^\gamma\cdot\mathbf{n}=0 \ \ \mbox{in $\partial\bk\times(0,T)$},\label{inf3}\\
(\uok,\utk)\mid_{t=0}&=&(\uo^{(0)}(x),\ut^{(0)}(x))\ \ \mbox{on $\bk$}.\label{inf4}
\end{eqnarray}
Of course, here
\begin{eqnarray*}
\wk&=&\uok+\utk,\\
\rok&=&\uok\fo(\wk)+\utk\go(\wk),\\
\rtk&=&\uok\ft(\wk)+\utk\gt(\wk).
\end{eqnarray*}
Moreover,
\begin{equation}
\uok\geq 0,\ \ \utk\geq 0,\ \ \wk\leq w_p.
\end{equation}
Adding \eqref{inf2} to \eqref{inf1} yields
\begin{equation}\label{wk}
\pt\wk-\frac{\gamma}{\gamma+1}\Delta\wkg^{\gamma+1}=\rok+\rtk\equiv\rk  \ \mbox{in $\bk\times(0,T)$}.
\end{equation}
Pick a smooth cut-off function $\zeta$ on $\RN$. From here on, we assume that $k$ is so large that
\begin{equation}
\textup{supp}\ \zeta\subset\bk.
\end{equation}
By using $\zeta^2\ln(\wk+\tau), \tau>0,$ as a test function in \eqref{wk}, we can infer from the proof of Lemma \ref{l41} that
\begin{eqnarray}
\int_{\RN\times(0,T)}\zeta^2\left|\nabla\weg^{\frac{\gamma+1}{2}}\right|^2dxdt\leq
c.
\end{eqnarray}      
Here $c$ depends on both $T$ and $\zeta$. Similarly, use $\zeta^2\pt\wkg^{\gamma+1}$ as test function in \eqref{wk} to derive
\begin{eqnarray}
\int_{\RN\times(0,T)}\zeta^2\left(\pt\wkg^{\frac{\gamma+2}{2}}\right)^2dxdt+\sup_{0\leq t\leq T}\int_{\RN}\zeta^2\left|\nabla \wkg^{\gamma+1}\right|^2dx\leq c.
\end{eqnarray}  
It is not difficult to see that Lemmas \ref{l32} and \ref{l33} still hold with $L^p(\ot) $ being replaced by $L^p(0,T; L^p_{\textup{loc}}(\RN)) $ and $L^2(0, T; W^{1,2}(\Omega))$ by $L^2(0, T; W^{1,2}_{\textup{loc}}(\RN))$, respectively. Take $k\rightarrow \infty$ in \eqref{inf1}-\eqref{inf4} suitably to conclude the corollary.         

\end{document}